\documentclass[12pt,reqno]{article}

\usepackage[usenames]{color}
\usepackage{amssymb}
\usepackage{graphicx}
\usepackage{amscd}

\usepackage{amsthm}
\newtheorem{theorem}{Theorem}
\newtheorem{lemma}[theorem]{Lemma}
\newtheorem{proposition}[theorem]{Proposition}
\newtheorem{corollary}[theorem]{Corollary}

\theoremstyle{definition}

\newtheorem{example}[theorem]{Example}

\usepackage[colorlinks=true,
linkcolor=webgreen, filecolor=webbrown,
citecolor=webgreen]{hyperref}

\definecolor{webgreen}{rgb}{0,.5,0}
\definecolor{webbrown}{rgb}{.6,0,0}

\usepackage{color}

\usepackage{float}

\usepackage{graphics,amsmath,amssymb}
\usepackage{amsfonts}
\usepackage{latexsym}
\usepackage{epsf}

\newcommand{\seqnum}[1]{\href{http://www.research.att.com/cgi-bin/access.cgi/as/~njas/sequences/eisA.cgi?Anum=#1}{\underline{#1}}}

\begin{document}

\begin{center}
\vskip 1cm{\LARGE\bf On the Halves of a Riordan Array and Their Antecedents} \vskip 1cm \large
Paul Barry\\
School of Science\\
Waterford Institute of Technology\\
Ireland\\
\href{mailto:pbarry@wit.ie}{\tt pbarry@wit.ie}
\end{center}
\vskip .2 in

\begin{abstract} Every Riordan array has what we call a horizontal half and a vertical half. These halves of a Riordan array have been studied separately before. Here, we place them in a common context, showing that one may be obtained from the other. Using them, we provide a canonical factorization of elements of the associated or Lagrange subgroup of the Riordan group. The vertical half matrix is shown to be an element of the hitting-time group. We also ask and answer the question: given a Riordan array, when is it the half (either horizontal of vertical) of a Riordan array?\end{abstract}

\section{Preliminaries on Riordan arrays}
We recall some facts about Riordan arrays in this introductory section. Readers familiar with Riordan arrays may wish to move on to the next section. A Riordan array is defined by a pair of power series
$$g(x)=g_0 + g_1 x + g_2 x^2 + \cdots=\sum_{n=0}^{\infty} g_n x^n,$$ and
$$f(x)=f_1 x + f_2 x^2+ f_3 x^3 + \cdots = \sum_{n=1}^{\infty} f_n x^n.$$
We require that $g_0 \ne 0$ (and hence $g(x)$ be invertible, with inverse $\frac{1}{g(x)}$), while we also demand that
$f_0=0$ and $f_1 \ne 0$ (hence $f(x)$ has a compositional inverse $\bar{f}(x)=\text{Rev}(f)(x)$ defined by $f(\bar{f}(x))=x$).
The set of such pairs $(g(x), f(x))$ forms a group (called the Riordan group \cite{SGWW}) with multiplication
$$(g(x), f(x)) \cdot (u(x), v(x))=(g(x)u(f(x)), v(f(x)),$$ and with inverses given by
$$(g(x), f(x))^{-1}=\left(\frac{1}{g(\bar{f}(x))}, \bar{f}(x)\right).$$
The coefficients of the power series may be drawn from any ring (for example, the integers $\mathbb{Z}$) where these operations make sense. To each such ring there exists a corresponding Riordan group.

There is a matrix representation of this group, where to the element $(g(x), f(x))$ we associate the matrix
$\left(a_{n,k}\right)_{0 \le n,k \le \infty}$ with general element
$$a_{n,k}=[x^n] g(x)f(x)^k.$$
Here, $[x^n]$ is the functional that extracts the coefficient of $x^n$ in a power series. In this representation, the group law corresponds to ordinary matrix multiplication, and the inverse of $(g(x), f(x))$ is represented by the inverse of $\left(a_{n,k}\right)_{0 \le n,k \le \infty}$.

The Fundamental Theorem of Riordan arrays is the rule
$$(g(x), f(x))\cdot h(x)=g(x)h(f(x)),$$ detailing how an array $(g(x), f(x))$ can act on a power series. This corresponds to the matrix $(a_{n,k})$ multiplying the vector $(h_0, h_1, h_2,\ldots)^T$.

\begin{example} Pascal's triangle, also known as the binomial matrix, is defined by the Riordan group element
$$\left(\frac{1}{1-x}, \frac{x}{1-x}\right).$$ This means that we have
$$\binom{n}{k}=[x^n] \frac{1}{1-x} \left(\frac{x}{1-x}\right)^k.$$
To see that this is so, we need to be familiar with the rules of operation of the functional $[x^n]$ \cite{MC}.
We have
\begin{align*}
[x^n] \frac{1}{1-x} \left(\frac{x}{1-x}\right)^k&=[x^n] \frac{x^k}{(1-x)^{k+1}}\\
&= [x^{n-k}] (1-x)^{-(k+1)}\\
&= [x^{n-k}] \sum_{j=0}^{\infty} \binom{-(k+1)}{j}(-1)^j x^j\\
&= [x^{n-k}] \sum_{j=0}^{\infty} \binom{k+1+j-1}{j}x^j\\
&= [x^{n-k}] \sum_{j=0}^{\infty} \binom{k+j}{j} x^j \\
&= \binom{k+n-k}{n-k}=\binom{n}{n-k}=\binom{n}{k}.\end{align*}
\end{example}

The binomial matrix is an element of the Bell subgroup of the Riordan group, consisting of arrays of the form $(g(x), xg(x))$. It is also an element of the hitting time subgroup, which consists of arrays of the form
$\left(\frac{x f'(x)}{f(x)}, f(x)\right)$. Arrays of the form $(1, f(x))$ belong to the associated or Lagrange subgroup of the Riordan group.

Note that all the arrays in this note are lower triangular matrices of infinite extent. We show appropriate truncations.

Many examples of sequences and  Riordan arrays are documented in the On-Line Encyclopedia of Integer Sequences (OEIS) \cite{SL1, SL2}. Sequences are frequently referred to by their
OEIS number. For instance, the binomial matrix $\mathbf{B}=\left(\frac{1}{1-x}, \frac{x}{1-x}\right)$ (``Pascal's triangle'') is \seqnum{A007318}. In the sequel we will not distinguish between an array pair $(g(x), f(x))$ and its matrix representation.

Having written this note, the author discovered the paper \cite{Bala} by Peter Bala, written originally in 2015, which covers some of the ground of this paper and can be read as a complementary text.

\section{The vertical and horizontal halves of a Riordan array}
Given a Riordan array $M=(g(x), f(x))=(g(x), xh(x))$ with matrix representation $\left(T_{n,k}\right)$ we shall denote by its \emph{vertical half} the matrix $M_V$ with general $(n,k)$-th term $T_{2n-k,n}$.
We have the following result \cite{Yang2, Yang1}.
\begin{lemma}
Given a Riordan array $M=(g(x), f(x))$, its vertical half $V$ is the Riordan array
$$V=\left(\frac{\phi(x)\phi'(x) g(\phi(x))}{f(\phi(x))}, \phi(x)\right)=\left(\frac{x \phi'(x) g(\phi(x))}{\phi(x)}, \phi(x)\right),$$ where
$$\phi(x)=\text{Rev}\left(\frac{x^2}{f(x)}\right).$$
\end{lemma}
\begin{proof}
We have
\begin{align*}T_{2n-k,n}&=[x^{2n-k}]g(x)(xh(x))^n\\
&=[x^n]g(x)x^k h^n\\
&=(n+1) \frac{1}{n+1} [x^n] \frac{gx^k}{h} h^{n+1}\\
&=(n+1) [x^{n+1}] G\left(\text{Rev}\left(\frac{x}{h}\right)\right)\\
&=[x^n] G'\left(\text{Rev}\left(\frac{x}{h}\right)\right) \frac{d}{dx}\text{Rev}\left(\frac{x}{h}\right)\\
&=[x^n] \left(\frac{gx^k}{h}\right)(\phi) \phi'\\
&=[x^n] \frac{x \phi' g(\phi)}{\phi} \phi^k,\end{align*}
where we have used the fact that $h(\phi)=\frac{\phi}{x}$.
\end{proof}
\begin{corollary} We have the factorization
$$V=(g(\phi(x)),x) \cdot \left(\frac{x \phi'(x)}{\phi(x)}, \phi(x)\right),$$ where the factor
$\left(\frac{x \phi'(x)}{\phi(x)}, x\right)$ is an element of the hitting-time subgroup of the Riordan group.
\end{corollary}
\begin{proof}
We have $\phi(x)=\text{Rev}\left(\frac{x^2}{f(x)}\right)$, and so $f(x)=\frac{x^2}{\bar{\phi}(x)}$. Hence we have
$$f(\phi(x))=\frac{\phi(x)^2}{x},$$ and so
\begin{align*}\frac{\phi(x)\phi'(x) g(\phi(x))}{f(\phi(x))}&=\frac{\phi(x)\phi'(x)}{\frac{\phi^2}{x}} g(\phi(x))\\
&=\frac{x \phi'(x)}{\phi(x)}.g(\phi(x)).\end{align*}
Thus
$$V=\left(\frac{x \phi'(x)}{\phi(x)}.g(\phi(x)), \phi(x)\right)=(g(\phi(x)),x)\cdot \left(\frac{x \phi'(x)}{\phi(x)},\phi(x)\right).$$
\end{proof}
The \emph{horizontal half} $H$ of the array $M=(g(x), f(x))$ is the array whose matrix representation has general $(n,k)$-th term given by $T_{2n,n+k}$.
We then have the following result \cite{Central}.
\begin{lemma}Given a Riordan array $(g(x), f(x))$, its horizontal half $H$ is the Riordan array
$$\left(\frac{\phi(x)\phi'(x) g(\phi(x))}{f(\phi(x))}, f(\phi(x))\right),$$ where
$$\phi(x)=\text{Rev}\left(\frac{x^2}{f(x)}\right).$$
\end{lemma}
For completeness, we reproduce the proof from \cite{Central}
\begin{proof}
The matrix $H$ is a Riordan array since we exhibit it as the product of two Riordan arrays. In order to show that it is the product of two Riordan arrays, we proceed as follows, using Lagrange inversion.
\begin{eqnarray*}
T_{2n,n+k}&=& [x^{2n}] g(x)(xf(x))^{n+k}\\
&=&[x^{2n}] x^{n+k}g(x)f(x)^k f(x)^n\\
&=&[x^n] g(x)(xf(x))^k f(x)^n\\
&=&\sum_{i=0}^n [x^i] g(x)(xf(x))^k [x^{n-i}]f(x)^n\\
&=& \sum_{i=0}^n a_{i,k} [x^n] \frac{x^i}{f(x)} f(x)^{n+1}\\
&=& \sum_{i=0}^n a_{i,k} (n+1) \frac{1}{n+1} [x^n] F'(x) f(x)^{n+1} \quad (F'(x)=\frac{x^i}{f(x)})\\
&=& \sum_{i=0}^n a_{i,k} (n+1) [x^{n+1}] F\left(\textrm{Rev}\left(\frac{x}{f(x)}\right)\right)\\
&=& \sum_{i=0}^n a_{i,k} [x^n] F'\left(\textrm{Rev}\left(\frac{x}{f(x)}\right)\right)\frac{d}{dx} \textrm{Rev}\left(\frac{x}{f(x)}\right)\\
&=& \sum_{i=0}^n a_{i,k} [x^n] \frac{\left(\textrm{Rev}\left(\frac{x}{f(x)}\right)\right)^i}{f\left(\textrm{Rev}\left(\frac{x}{f(x)}\right)\right)} \frac{d}{dx} \textrm{Rev}\left(\frac{x}{f(x)}\right)\\
&=& \sum_{i=0}^n T_{i,k} [x^n] \frac{\phi'(x)}{f(\phi(x))} (\phi(x))^i\\
&=& \sum_{i=0}^n m_{n,i} T_{i,k}, \end{eqnarray*}
where
$$m_{n,k}=[x^n] \frac{\phi'(x)}{f(\phi(x))}(\phi(x))^k$$ is the general term of the Riordan array
$$\left(\frac{\phi'}{f(\phi)}, \phi\right).$$
\end{proof}
\begin{corollary}
We have
$$H=(g(\phi(x)),x) \cdot \left(\frac{x \phi'(x)}{\phi(x)}, f(\phi(x))\right).$$
\end{corollary}

\begin{example}
We consider the Pascal-like Riordan array $\left(\frac{1}{1-x}, \frac{x(1+x)}{1-x}\right)$ which begins
$$\left(
\begin{array}{ccccccc}
 1 & 0 & 0 & 0 & 0 & 0 & 0 \\
 1 & 1 & 0 & 0 & 0 & 0 & 0 \\
 1 & 3 & 1 & 0 & 0 & 0 & 0 \\
 1 & 5 & 5 & 1 & 0 & 0 & 0 \\
 1 & 7 & 13 & 7 & 1 & 0 & 0 \\
 1 & 9 & 25 & 25 & 9 & 1 & 0 \\
 1 & 11 & 41 & 63 & 41 & 11 & 1 \\
\end{array}
\right).$$
The elements of the vertical half of this array are in bold in the following.
$$\left(
\begin{array}{ccccccc}
 \mathbf{1} & 0 & 0 & 0 & 0 & 0 & 0 \\
 1 & \mathbf{1} & 0 & 0 & 0 & 0 & 0 \\
 1 & \mathbf{3} & \mathbf{1} & 0 & 0 & 0 & 0 \\
 1 & 5 & \mathbf{5} & \mathbf{1} & 0 & 0 & 0 \\
 1 & 7 & \mathbf{13} & \mathbf{7} & 1 & 0 & 0 \\
 1 & 9 & 25 & \mathbf{25} & 9 & 1 & 0 \\
 1 & 11 & 41 & \mathbf{63} & 41 & 11 & 1 \\
\end{array}
\right).$$
The elements of the horizontal half of this array are in bold in the following.
$$A=\left(\begin{array}{ccccccc} \mathbf{1} & 0 & 0 & 0
&0 & 0 & \cdots \\1 & 1 & 0 & 0 & 0 & 0 & \cdots \\ 1 & \mathbf{3} & \mathbf{1} &
0 & 0 & 0 &
\cdots \\ 1 & 5 & 5 & 1 & 0 & 0 & \cdots \\ 1 & 7 & \mathbf{13} & \mathbf{7} &
\mathbf{1} & 0 & \cdots \\1 & 9  & 25 & 25 & 9 & 1 &\cdots\\ \vdots
& \vdots &
\vdots & \vdots & \vdots & \vdots &
\ddots\end{array}\right).$$
We have
$$\phi(x)=\text{Rev}\left(\frac{x^2}{f(x)}\right)=\text{Rev}\left(\frac{x(1-x)}{1+x}\right)=\frac{1-x-\sqrt{1-6x+x^2}}{2}.$$
Then
$$\frac{\phi(x)\phi'(x) g(\phi(x))}{f(\phi(x))}=\frac{1}{\sqrt{1-6x+x^2}}$$ and hence the vertical half array corresponding to $\left(\frac{1}{1-x}, \frac{x(1+x)}{1-x}\right)$ is given by the array
$$\left(\frac{1}{\sqrt{1-6x+x^2}}, \frac{1-x-\sqrt{1-6x+x^2}}{2}\right).$$
This matrix begins
$$V=\left(
\begin{array}{ccccccc}
 1 & 0 & 0 & 0 & 0 & 0 & 0 \\
 3 & 1 & 0 & 0 & 0 & 0 & 0 \\
 13 & 5 & 1 & 0 & 0 & 0 & 0 \\
 63 & 25 & 7 & 1 & 0 & 0 & 0 \\
 321 & 129 & 41 & 9 & 1 & 0 & 0 \\
 1683 & 681 & 231 & 61 & 11 & 1 & 0 \\
 8989 & 3653 & 1289 & 377 & 85 & 13 & 1 \\
\end{array}
\right).$$
Now
$$ f(\phi(x))=\frac{1-4x+x^2-(1-x)\sqrt{1-6x+x^2}}{2x},$$ and so the horizontal half of the array is given by
$$\left(\frac{1}{\sqrt{1-6x+x^2}}, \frac{1-4x+x^2-(1-x)\sqrt{1-6x+x^2}}{2x}\right).$$ This array begins
$$H=\left(
\begin{array}{ccccccc}
 1 & 0 & 0 & 0 & 0 & 0 & 0 \\
 3 & 1 & 0 & 0 & 0 & 0 & 0 \\
 13 & 7 & 1 & 0 & 0 & 0 & 0 \\
 63 & 41 & 11 & 1 & 0 & 0 & 0 \\
 321 & 231 & 85 & 15 & 1 & 0 & 0 \\
 1683 & 1289 & 575 & 145 & 19 & 1 & 0 \\
 8989 & 7183 & 3649 & 1159 & 221 & 23 & 1 \\
\end{array}
\right).$$
We calculate the product
$$V^{-1} \cdot H.$$
We get the matrix that begins
$$\left(
\begin{array}{ccccccc}
 1 & 0 & 0 & 0 & 0 & 0 & 0 \\
 0 & 1 & 0 & 0 & 0 & 0 & 0 \\
 0 & 2 & 1 & 0 & 0 & 0 & 0 \\
 0 & 2 & 4 & 1 & 0 & 0 & 0 \\
 0 & 2 & 8 & 6 & 1 & 0 & 0 \\
 0 & 2 & 12 & 18 & 8 & 1 & 0 \\
 0 & 2 & 16 & 38 & 32 & 10 & 1 \\
\end{array}
\right).$$
This is the array $(1, f(x))$. This is a general result.
\end{example}
\begin{proposition} Let $V$ and $H$ be respectively the vertical and horizontal halves of the Riordan array $(g(x), f(x))$. Then we have
$$V^{-1} \cdot H = (1, f(x)).$$
\end{proposition}
\begin{proof} We have
$$\left(\frac{\phi(x)\phi'(x) g(\phi(x))}{f(\phi(x))}, f(\phi(x))\right)=\left(\frac{\phi(x)\phi'(x) g(\phi(x))}{f(\phi(x))}, \phi(x)\right)\cdot (1, f(x)).$$
\end{proof}
In general, we have
$$V=(g(\phi), x)\cdot \left(\frac{x \phi'}{\phi}, \phi\right),$$
$$H=(g(\phi), x)\cdot \left(\frac{x \phi'}{\phi}, f(\phi)\right)=(g(\phi), x)\cdot \left(\frac{x \phi'}{\phi}, \phi\right)\cdot (1, f).$$
Thus
$$H = V \cdot (1,f).$$

We end this section with a generic factorization of elements of the associated or Lagrange subgroup of the Riordan group.
\begin{proposition}
Let $A=(1, f(x))$ be an element of the associated group. Let $H$ and $V$ respectively be the horizontal half and the vertical half of $A$. The array $V$ is an element of the hitting time subgroup, and thus so is $V^{-1}$. We then have
$$ V \cdot A = H,$$ or equivalently,
$$ A = V^{-1} \cdot H.$$
\end{proposition}
\begin{proof}
For $A=(1,f)$, we have $g(x)=1$, and so $(g(\phi),x)=(1,x)$ the identity. We then apply the above result.

We may also verify the factorization directly as follows. Since $\phi=\text{Rev}\left(\frac{x^2}{f(x)}\right)$, we have that $f(x)=\frac{x^2}{\bar{\phi}(x)}$ and so $f(\phi(x))=\frac{\phi(x)^2}{x}$.  Then the result corresponds to the following product of Riordan arrays.
$$\left(\frac{x\phi'(x)}{\phi(x)}, \phi(x)\right)\cdot \left(1, \frac{x^2}{\bar{\phi}(x)}\right)=\left(\frac{x\phi'(x)}{\phi(x)}, \frac{\phi^2}{x}\right),$$ which is easily verified.
\end{proof}

\begin{example}
We seek the ``hitting-time'' factorization of the Catalan matrix given by $(1, xc(x))$ where $c(x)=\frac{1-\sqrt{1-4x}}{2x}$ is the generating function of the Catalan numbers $\frac{1}{n+1} \binom{2n}{n}$.

This array begins
$$\left(
\begin{array}{ccccccc}
 1 & 0 & 0 & 0 & 0 & 0 & 0 \\
 0 & 1 & 0 & 0 & 0 & 0 & 0 \\
 0 & 1 & 1 & 0 & 0 & 0 & 0 \\
 0 & 2 & 2 & 1 & 0 & 0 & 0 \\
 0 & 5 & 5 & 3 & 1 & 0 & 0 \\
 0 & 14 & 14 & 9 & 4 & 1 & 0 \\
 0 & 42 & 42 & 28 & 14 & 5 & 1 \\
\end{array}
\right).$$
We have
$$f(x)=xc(x) \Longrightarrow \phi(x)=\frac{2 \sqrt{x} \sin \left(\frac{1}{3} \sin ^{-1}\left(\frac{3 \sqrt{3}
   \sqrt{x}}{2}\right)\right)}{\sqrt{3}}.$$
We then find that
$$f(\phi(x))=\frac{1}{2}-\frac{\sqrt{\sqrt{3}-8 \sqrt{x} \sin \left(\frac{1}{3} \sin ^{-1}\left(\frac{3 \sqrt{3}
   \sqrt{x}}{2}\right)\right)}}{2 \sqrt[4]{3}},$$ and
$$\frac{x \phi'}{\phi}=\frac{\sqrt{3} \sqrt{x} \cot \left(\frac{1}{3} \sin ^{-1}\left(\frac{3 \sqrt{3}
   \sqrt{x}}{2}\right)\right)}{2 \sqrt{4-27 x}}+\frac{1}{2}.$$
The vertical half $V$ of $(1, xc(x))$ then begins
$$\left(
\begin{array}{ccccccc}
 1 & 0 & 0 & 0 & 0 & 0 & 0 \\
 1 & 1 & 0 & 0 & 0 & 0 & 0 \\
 5 & 2 & 1 & 0 & 0 & 0 & 0 \\
 28 & 9 & 3 & 1 & 0 & 0 & 0 \\
 165 & 48 & 14 & 4 & 1 & 0 & 0 \\
 1001 & 275 & 75 & 20 & 5 & 1 & 0 \\
 6188 & 1638 & 429 & 110 & 27 & 6 & 1 \\
\end{array}
\right).$$ This is
$$\left(\frac{\sqrt{3} \sqrt{x} \cot \left(\frac{1}{3} \sin ^{-1}\left(\frac{3 \sqrt{3}
   \sqrt{x}}{2}\right)\right)}{2 \sqrt{4-27 x}}+\frac{1}{2}, \frac{2 \sqrt{x} \sin \left(\frac{1}{3} \sin ^{-1}\left(\frac{3 \sqrt{3}
   \sqrt{x}}{2}\right)\right)}{\sqrt{3}}\right).$$
The horizontal half $H$ of $(1, xc(x))$ begins
$$\left(
\begin{array}{ccccccc}
 1 & 0 & 0 & 0 & 0 & 0 & 0 \\
 1 & 1 & 0 & 0 & 0 & 0 & 0 \\
 5 & 3 & 1 & 0 & 0 & 0 & 0 \\
 28 & 14 & 5 & 1 & 0 & 0 & 0 \\
 165 & 75 & 27 & 7 & 1 & 0 & 0 \\
 1001 & 429 & 154 & 44 & 9 & 1 & 0 \\
 6188 & 2548 & 910 & 273 & 65 & 11 & 1 \\
\end{array}
\right).$$ This is
$$\left(\frac{\sqrt{3} \sqrt{x} \cot \left(\frac{1}{3} \sin ^{-1}\left(\frac{3 \sqrt{3}
   \sqrt{x}}{2}\right)\right)}{2 \sqrt{4-27 x}}+\frac{1}{2},\frac{1}{2}-\frac{\sqrt{\sqrt{3}-8 \sqrt{x} \sin \left(\frac{1}{3} \sin ^{-1}\left(\frac{3 \sqrt{3}
   \sqrt{x}}{2}\right)\right)}}{2 \sqrt[4]{3}}\right).$$
\end{example}

We note that we can express the inverse $V^{-1}$ of the vertical half $V$ in terms of $f(x)$.
\begin{proposition} We have
$$ V^{-1}=\left(2-\frac{xf'(x)}{f(x)}, \frac{x^2}{f(x)}\right).$$
\end{proposition}
\begin{proof}
We have $$V=\left(\frac{x \phi'(x)}{\phi(x)}, \phi(x)\right),$$ where
$$\phi(x)=\text{Rev}\left(\frac{x^2}{f(x)}\right).$$
Now since $V$ is in the hitting-time subgroup, its inverse will be given by
$$V^{-1}=\left(\frac{x \bar{\phi}'(x)}{\bar{\phi}(x)}, \bar{\phi}(x)\right).$$
Here, we have
$$\bar{\phi}(x)=\frac{x^2}{f(x)}.$$
We find that
$$\frac{x \bar{\phi}'(x)}{\bar{\phi}(x)}=\frac{2f(x)-xf'(x)}{f(x)},$$ and the result follows.
\end{proof}

We close this section by asking the question: what condition on $f(x)$ guarantees that the vertical half $V$ be a pseudo-involution? The following provides an answer.

\begin{proposition} The vertical half $V$ of the array $(1, f(x))$ will be a pseudo-involution whenever we have
$$\frac{x^2}{f(x)}=\text{Rev}\left(\frac{x^2}{-f(-x)}\right).$$
\end{proposition}
\begin{proof} In order that $V=\left(\frac{x \phi'(x)}{\phi(x)}, \phi(x)\right)$ be a pseudo-involution, we require that
$$\bar{\phi}(x)=-\phi(-x).$$
But since $\bar{\phi}(x)=\frac{x^2}{f(x)}$, this means that we require
$$\frac{x^2}{f(x)}=-\text{Rev}\left(\frac{x^2}{f(-x)}\right)=\text{Rev}\left(\frac{x^2}{-f(-x)}\right).$$
\end{proof}
\begin{example}
We take $f(x)=x(1+x)$. Then $(1, f(x))$ is the matrix $\left(\binom{n-1}{n-k}\right)$. Now
$$\phi(x)=\text{Rev}\left(\frac{x^2}{x(1+x)}\right)=\frac{x}{1-x}.$$
This means that $$V=\left(\frac{1}{1-x}, \frac{x}{1-x}\right).$$
This is the binomial matrix $\left(\binom{n}{k}\right)$, which is indeed a pseudo-involution.
\end{example}
We close this section by noting the importance of the hitting-time group element $\left(\frac{x \phi'}{\phi}, \phi\right)$.
$$ H = (g(\phi),x)\cdot \left(\frac{x \phi'}{\phi}, \frac{\phi^2}{x}\right)=\left(\frac{x \phi'}{\phi}, \phi\right)\cdot (g, f).$$
$$V=(g(\phi), x) \cdot \left(\frac{x \phi'}{\phi}, \phi\right).$$ 
We see that in general we have 
$$ V \cdot (g, f)= (g(\phi(x)),x) \cdot H.$$
\section{Riordan antecedents of a Riordan array half}
We now consider the question: given a Riordan array $(\psi(x), \phi(x))$, can we express it as the half (either vertical or horizontal) of another Riordan array $(g(x), f(x))$? If such a Riordan array $(g(x), f(x))$ exists, we shall call it a \emph{Riordan antecedent} of  the given matrix $(\psi(x), \phi(x))$. Note that if we drop the stipulation of being Riordan, then it is possible to define an arbitrary number of lower-triangular invertible antecedents.
\begin{example} Consider the array $R=(\psi(x), \phi(x))=\left(\frac{1}{1-x}, \frac{x}{1-x}\right)$ (the binomial array). As we have seen, its general element is $\binom{n}{k}$. We  claim that the matrix with general term $M=\binom{k}{2k-n}$ is a vertical antecedent  of $A$, that is, we have $M_V=A$. This follows since if $T_{n,k}=\binom{k}{2k-n}$, then $T_{2n-k,n}=\binom{n}{2n-(2n-k)}=\binom{n}{k}$.
The matrix $(\binom{k}{2k-n})=(\binom{k}{n-k})$ begins
$$\left(
\begin{array}{cccccccc}
 1 & 0 & 0 & 0 & 0 & 0 & 0 & 0 \\
 0 & 1 & 0 & 0 & 0 & 0 & 0 & 0 \\
 0 & 1 & 1 & 0 & 0 & 0 & 0 & 0 \\
 0 & 0 & 2 & 1 & 0 & 0 & 0 & 0 \\
 0 & 0 & 1 & 3 & 1 & 0 & 0 & 0 \\
 0 & 0 & 0 & 3 & 4 & 1 & 0 & 0 \\
 0 & 0 & 0 & 1 & 6 & 5 & 1 & 0 \\
 0 & 0 & 0 & 0 & 4 & 10 & 6 & 1 \\
\end{array}
\right).$$
This is in fact the Riordan array $(1, x(1+x))$ so in this case we see that the binomial matrix has a vertical Riordan antecedent.
\end{example}

\begin{example} For this example, we again take the case of the binomial matrix $\left(\binom{n}{k}\right)$.
We now consider the matrix with general term $\binom{\frac{n}{2}}{k-\frac{n}{2}}$. This matrix begins
$$\left(
\begin{array}{cccccccc}
 1 & 0 & 0 & 0 & 0 & 0 & 0 & 0 \\
 \frac{1}{2} & 1 & 0 & 0 & 0 & 0 & 0 & 0 \\
 0 & 1 & 1 & 0 & 0 & 0 & 0 & 0 \\
 -\frac{1}{16} & \frac{3}{8} & \frac{3}{2} & 1 & 0 & 0 & 0 & 0 \\
 0 & 0 & 1 & 2 & 1 & 0 & 0 & 0 \\
 \frac{3}{256} & -\frac{5}{128} & \frac{5}{16} & \frac{15}{8} & \frac{5}{2} & 1 & 0 & 0 \\
 0 & 0 & 0 & 1 & 3 & 3 & 1 & 0 \\
 -\frac{5}{2048} & \frac{7}{1024} & -\frac{7}{256} & \frac{35}{128} & \frac{35}{16} & \frac{35}{8} & \frac{7}{2} &
   1 \\
\end{array}
\right).$$
Letting $T_{n,k}=\binom{\frac{n}{2}}{k-\frac{n}{2}}$, we can calculate that $T_{2n,n+k}=\binom{n}{k}$. Thus the matrix $(\binom{\frac{n}{2}}{k-\frac{n}{2}})$ is a horizontal antecedent of the binomial matrix.
\end{example}

We have the following general results.
\begin{proposition} Every Riordan array $(\Psi(x), \Phi(x))$ has a vertical Riordan antecedent.
\end{proposition}
\begin{proof} If such an antecedent exists, then there would exist a Riordan array $(g(x), f(x))$ with vertical half given by
$$\left(\frac{\phi(x)\phi'(x) g(\phi(x))}{f(\phi(x))}, \phi(x)\right),$$ where
$$\phi(x)=\text{Rev}\left(\frac{x^2}{f(x)}\right).$$ We would then have
$$\Phi(x)=\phi(x)$$ and
$$\frac{\phi(x)\phi'(x) g(\phi(x))}{f(\phi(x))}=\Psi(x).$$
Thus we must show that we can solve these last two equations to obtain $g(x)$ and $f(x)$.
Using the presumed equality of $\Phi(x)$ and $\phi(x)$, we set
$$\Phi(x)=\text{Rev}\left(\frac{x^2}{f(x)}\right),$$ or
$$ f(x) = \frac{x^2}{\text{Rev}(\Phi(x))}.$$
We now solve the equation
$$\frac{\Phi'\cdot \Phi\cdot u}{f(\Phi)}=\Psi$$ to obtain
$u=u(x)$.
Then setting $g(x)=u(\text{Rev}(\Phi(x)))$ gives us the second element of the pair $(g(x), f(x))$.
\end{proof}
\begin{example}
We consider the Riordan array $(\psi, \phi)=\left(\frac{1}{1-x}, \frac{x(1+x)}{1-x}\right)$.
We have
$$\phi(x)=\frac{x(1+x)}{1-x} \Longrightarrow \bar{\phi}(x)=\frac{\sqrt{1+6x+x^2}-x-1}{2}.$$
Thus we have
$$f(x)=\frac{x^2}{\bar{\phi}(x)}=\frac{x(1+x+\sqrt{1+6x+x^2})}{2}.$$
We calculate $f(\phi(x))=\frac{x(1+x)^2}{(1-x)^2}$. We must now solve for $u$ in the equation
$$\frac{\phi'\cdot \phi\cdot u}{f(\phi)}=\psi=\frac{1}{1-x}.$$
We obtain $$u(x)=\frac{1+x}{1+2x-x^2}.$$
Finally we have
$$g(x)=u(\bar{\phi}(x))=\frac{1+x+\sqrt{1+6x+x^2}}{2 \sqrt{1+6x+x^2}}.$$
The matrix $$\left(\frac{1+x+\sqrt{1+6x+x^2}}{2 \sqrt{1+6x+x^2}}, \frac{x(1+x+\sqrt{1+6x+x^2})}{2}\right)$$ is thus a vertical Riordan antecedent of the Pascal-like Riordan array $\left(\frac{1}{1-x}, \frac{x(1+x)}{1-x}\right)$. This antecedent matrix begins
$$\left(
\begin{array}{cccccccc}
 1 & 0 & 0 & 0 & 0 & 0 & 0 & 0 \\
 -1 & 1 & 0 & 0 & 0 & 0 & 0 & 0 \\
 5 & 1 & 1 & 0 & 0 & 0 & 0 & 0 \\
 -25 & 1 & 3 & 1 & 0 & 0 & 0 & 0 \\
 129 & -7 & 1 & 5 & 1 & 0 & 0 & 0 \\
 -681 & 41 & -1 & 5 & 7 & 1 & 0 & 0 \\
 3653 & -231 & 9 & 1 & 13 & 9 & 1 & 0 \\
 -19825 & 1289 & -61 & 1 & 7 & 25 & 11 & 1 \\
\end{array}
\right).$$
\end{example}
We also have the following result.
\begin{proposition} Assume given a Riordan array $(\Psi(x), \Gamma(x))$. If there exists a solution $f(x)$ of the implicit equation
$$\Gamma\left(\frac{x^2}{f(x)}\right)=f(x)$$ with $f(0)=0$, then we can construct a horizontal Riordan antecedent of $(\Psi(x), \Gamma(x))$.
\end{proposition}
\begin{proof} If such an antecedent $(g(x), f(x))$ existed, then we should have
$$(\Psi(x), \Gamma(x))=\left(\frac{\phi(x)\phi'(x) g(\phi(x))}{f(\phi(x))}, f(\phi(x))\right),$$ where
$$\phi(x)=\text{Rev}\left(\frac{x^2}{f(x)}\right).$$
This last equation gives us
$$\bar{\phi}(x)=\frac{x^2}{f(x)}.$$
Now $\Gamma(x)=f(\phi(x))$ gives us that
$$ \Gamma(\bar{\phi}(x))=f(\phi(\bar{\phi}(x)))=f(x).$$
This gives us an implicit equation for $f$.
$$ \Gamma\left(\frac{x^2}{f(x)}\right)=f(x).$$
Let $f(x)$ be the solution of this equation for which $f(0)=0$.
We can now solve for the corresponding $\phi$ since $\phi(x)=\text{Rev}\left(\frac{x^2}{f(x)}\right)$.
As before, we can now solve for $g(x)$.
\end{proof}

\begin{example}
We consider the case of the binomial matrix $\left(\frac{1}{1-x}, \frac{x}{1-x}\right)$.
Thus $\Gamma(x)=\frac{x}{1-x}$, and the equation
 $$\Gamma\left(\frac{x^2}{f(x)}\right)=f(x)$$ becomes the equation
 $$\frac{x^2}{f-x^2}=f,$$ with solution
 $$f(x)=\frac{x^2+x \sqrt{x^2+4}}{2}.$$
This gives us
$$\bar{\phi}(x)=\frac{x^2}{f(x)}=\frac{x(\sqrt{x^2+4}-x)}{2}.$$
Solving for $\phi(x)$, we get
$$\phi(x)=\frac{x}{\sqrt{1-x}}.$$
Solving the equation
$$ \frac{\phi'\cdot \phi \cdot u}{f(\phi)}=\frac{1}{1-x},$$ we obtain
$u(x)=\frac{2}{2-x}$. Then we have
$$g(x)=u(\bar{\phi}(x))=\frac{x+\sqrt{x^2+4}}{\sqrt{x^2+4}}.$$
Thus the sought-after horizontal Riordan antecedent of the binomial matrix is given by the Riordan array
$$\left(\frac{x+\sqrt{x^2+4}}{\sqrt{x^2+4}}, \frac{x^2+x \sqrt{x^2+4}}{2}\right).$$
This matrix begins
$$\left(
\begin{array}{ccccccccccc}
 1 & 0 & 0 & 0 & 0 & 0 & 0 & 0 & 0 & 0 & 0 \\
 \frac{1}{2} & 1 & 0 & 0 & 0 & 0 & 0 & 0 & 0 & 0 & 0 \\
 0 & 1 & 1 & 0 & 0 & 0 & 0 & 0 & 0 & 0 & 0 \\
 -\frac{1}{16} & \frac{3}{8} & \frac{3}{2} & 1 & 0 & 0 & 0 & 0 & 0 & 0 & 0 \\
 0 & 0 & 1 & 2 & 1 & 0 & 0 & 0 & 0 & 0 & 0 \\
 \frac{3}{256} & -\frac{5}{128} & \frac{5}{16} & \frac{15}{8} & \frac{5}{2} & 1 & 0 & 0 & 0 & 0 & 0 \\
 0 & 0 & 0 & 1 & 3 & 3 & 1 & 0 & 0 & 0 & 0 \\
 -\frac{5}{2048} & \frac{7}{1024} & -\frac{7}{256} & \frac{35}{128} & \frac{35}{16} & \frac{35}{8} & \frac{7}{2} &
   1 & 0 & 0 & 0 \\
 0 & 0 & 0 & 0 & 1 & 4 & 6 & 4 & 1 & 0 & 0 \\
 \frac{35}{65536} & -\frac{45}{32768} & \frac{9}{2048} & -\frac{21}{1024} & \frac{63}{256} & \frac{315}{128} &
   \frac{105}{16} & \frac{63}{8} & \frac{9}{2} & 1 & 0 \\
 0 & 0 & 0 & 0 & 0 & 1 & 5 & 10 & 10 & 5 & 1 \\
\end{array}
\right).$$
We note that while the original matrix was integer valued, its antecedent has rational values. The inverse of the antecedent matrix is the Riordan array
$$\left(\frac{2+x}{2(1+x)}, \frac{x}{\sqrt{1+x}}\right).$$
\end{example}
\begin{example} We now consider the matrix $(\Psi(x), \Gamma(x))=\left(\frac{1}{1-x}, \frac{x}{(1-x)^2}\right)$. This matrix \seqnum{A085478} begins as follows.
$$\left(
\begin{array}{ccccc}
 1 & 0 & 0 & 0 & 0 \\
 1 & 1 & 0 & 0 & 0 \\
 1 & 3 & 1 & 0 & 0 \\
 1 & 6 & 5 & 1 & 0 \\
 1 & 10 & 15 & 7 & 1 \\
\end{array}
\right).$$ It has its general $(n,k)$ element given by $\binom{n+k}{2k}$.
In this case, the implicit equation
$$\Gamma\left(\frac{x^2}{f(x)}\right)=f(x)$$ becomes the equation
$$\frac{f x^2}{(x^2-f)^2} = f,$$ with solution
$$f(x)=x(1+x).$$
Thus
$$\bar{\phi}(x)=\frac{x^2}{f(x)}=\frac{x^2}{x(1+x)}=\frac{x}{1+x},$$ from which we get
$$\phi(x)=\frac{x}{1-x}.$$
Proceeding as above we now find that $$g(x)=1.$$
Thus the horizontal antecedent of $\left(\frac{1}{1-x}, \frac{x}{(1-x)^2}\right)$ is the integer valued matrix $(1, x(1+x))=\left(\binom{k}{n-k}\right)$. This matrix begins as follows.
$$\left(
\begin{array}{cccccccc}
\mathbf{1} & 0 & 0 & 0 & 0 & 0 & 0 & 0 \\
 0 & 1 & 0 & 0 & 0 & 0 & 0 & 0 \\
 0 & \mathbf{1} & \mathbf{1} & 0 & 0 & 0 & 0 & 0 \\
 0 & 0 & 2 & 1 & 0 & 0 & 0 & 0 \\
 0 & 0 & \mathbf{1} & \mathbf{3} &\mathbf{1} & 0 & 0 & 0 \\
 0 & 0 & 0 & 3 & 4 & 1 & 0 & 0 \\
 0 & 0 & 0 & \mathbf{1} & \mathbf{6} & \mathbf{5} & \mathbf{1} & 0 \\
 0 & 0 & 0 & 0 & 4 & 10 & 6 & 1 \\
\end{array}
\right).$$
Letting $T_{n,k}=\binom{n+k}{2k}$, we have that $T_{\frac{n}{2}, k-\frac{n}{2}}=\binom{k}{n-k}$, which is the general term of $(1, x(1+x))$.
\end{example}
This example show that the matrix $(\binom{k}{n-k})$ is a vertical Riordan antecedent of the binomial matrix $(\binom{n}{k})$ and a horizontal Riordan antecedent of the matrix $(\binom{n+k}{2k})$.

With regard to the first component $g(x)$ of an antecedent, we have the following result.
\begin{proposition} Suppose given a Riordan antecedent $(g(x), f(x))$ of a Riordan array $(\psi(x), \phi(x))$. Then we have
$$g(x)=\frac{f(x)}{x} \bar{\phi}'(x) \psi(\bar{\phi}(x)).$$
\end{proposition}
\begin{proof}
We have $g(x)=u(\bar{\phi}(x))$ where
$$\frac{\phi\cdot \phi'\cdot u}{f(\phi)}=\psi.$$
This gives
$$u = \frac{f(\phi) \psi}{\phi \cdot \phi'}$$ and so
$$g(x)=u(\bar{\phi}(x))=\frac{f(\phi(\bar{\phi}(x))) \psi(\bar{\phi}(x))}{\phi(\bar{\phi}(x)) \cdot \phi'(\bar{\phi}(x))}$$
The result now follows from $\phi(\bar{\phi}(x))=x$ and $\phi'(\bar{\phi}(x))=\frac{1}{\bar{\phi}'(x)}$.
\end{proof}
\begin{corollary} Let $\Psi(x)$ be a primitive of $\psi(x)$: $\Psi'(x)=\psi(x)$. Then in the circumstances above, we have $$g(x)=\frac{f(x)}{x} \frac{d}{dx} \Psi(\bar{\phi}(x)).$$
\end{corollary}
\begin{corollary} Let $(g_1(x), f(x))$ and $(g_2(x), f(x))$ be antecedents of the Riordan arrays $(\psi_1(x), \phi(x))$ and $(\psi_2(x), \phi(x))$, respectively. Then we have
$$\frac{g_1(x)}{g_2(x)}=\frac{\psi_1(\bar{\phi}(x))}{\psi_2(\bar{\phi}(x))}.$$
\end{corollary}

\section{Further antecedent examples}
\begin{example} We consider the combinatorially important case of the Catalan triangle \seqnum{A033184} given by the Riordan array $(c(x), x c(x))$ where $c(x)=\frac{1-\sqrt{1-4x}}{2x}=\text{Rev}(x(1-x))$ is the generating function of the Catalan numbers $C_n=\frac{1}{n+1}\binom{2n}{n}$. This matrix has general element $\frac{k+1}{n+1} \binom{2n-k}{n-k}$ and it begins as follows.
$$\left(
\begin{array}{cccccc}
 1 & 0 & 0 & 0 & 0 & 0 \\
 1 & 1 & 0 & 0 & 0 & 0 \\
 2 & 2 & 1 & 0 & 0 & 0 \\
 5 & 5 & 3 & 1 & 0 & 0 \\
 14 & 14 & 9 & 4 & 1 & 0 \\
 42 & 42 & 28 & 14 & 5 & 1 \\
\end{array}
\right).$$
We take $(\psi(x), \phi(x))=(c(x), xc(x))$ and we seek the vertical antecedent $(g(x), f(x))$ of this matrix.
We have
$$f(x)=\frac{x^2}{\text{Rev}(xc(x))}=\frac{x^2}{x(1-x)}=\frac{x}{1-x}.$$
We have $\phi'(x)=(xc(x))'=\frac{1}{\sqrt{1-4x}}$, while
$$f(\phi(x))=\frac{1-2x-\sqrt{1-4x}}{2x}.$$
Solving for $g(x)$, we obtain
$$g(x)=\frac{1-2x}{(1-x)^2}.$$
Thus a vertical Riordan antecedent of the Catalan matrix $(c(x), xc(x))$ is given by the Riordan array
$$\left(\frac{1-2x}{(1-x)^2}, \frac{x}{1-x}\right).$$
This matrix begins
$$\left(
\begin{array}{ccccccc}
 1 & 0 & 0 & 0 & 0 & 0 & 0 \\
 0 & 1 & 0 & 0 & 0 & 0 & 0 \\
 -1 & 1 & 1 & 0 & 0 & 0 & 0 \\
 -2 & 0 & 2 & 1 & 0 & 0 & 0 \\
 -3 & -2 & 2 & 3 & 1 & 0 & 0 \\
 -4 & -5 & 0 & 5 & 4 & 1 & 0 \\
 -5 & -9 & -5 & 5 & 9 & 5 & 1 \\
\end{array}
\right).$$
We find that the horizontal Riordan antecedent of $(c(x), xc(x))$ is given by
$$\left(\frac{2-5x+3x^2}{2-4x}, x\sqrt{1-x}\right)^{-1}.$$ This array begins
$$\left(
\begin{array}{cccccccc}
 1 & 0 & 0 & 0 & 0 & 0 & 0 & 0 \\
 \frac{1}{2} & 1 & 0 & 0 & 0 & 0 & 0 & 0 \\
 0 & 1 & 1 & 0 & 0 & 0 & 0 & 0 \\
 -\frac{21}{16} & \frac{7}{8} & \frac{3}{2} & 1 & 0 & 0 & 0 & 0 \\
 -5 & 0 & 2 & 2 & 1 & 0 & 0 & 0 \\
 -\frac{3861}{256} & -\frac{429}{128} & \frac{33}{16} & \frac{27}{8} & \frac{5}{2} & 1 & 0 & 0 \\
 -42 & -14 & 0 & 5 & 5 & 3 & 1 & 0 \\
 -\frac{230945}{2048} & -\frac{46189}{1024} & -\frac{2431}{256} & \frac{715}{128} & \frac{143}{16} & \frac{55}{8}
   & \frac{7}{2} & 1 \\
\end{array}
\right).$$
\end{example}
\begin{example} We finish with the case of the Riordan array $\left(\frac{1}{1-x-x^2},\frac{x}{1-x-x^2}\right)$. This matrix \seqnum{A037027} begins
$$\left(
\begin{array}{ccccc}
 1 & 0 & 0 & 0 & 0 \\
 1 & 1 & 0 & 0 & 0 \\
 2 & 2 & 1 & 0 & 0 \\
 3 & 5 & 3 & 1 & 0 \\
 5 & 10 & 9 & 4 & 1 \\
\end{array}
\right).$$
Letting $\Gamma(x)=\frac{x}{1-x-x^2}$, we find that
$$f(x)=\frac{1}{2}(x^2+x \sqrt{4+5x^2}).$$
Also we obtain $$\bar{\phi}(x)=\frac{2x^2}{x^2+x \sqrt{4+5x^2}}$$ and
$$\phi(x)=\frac{x}{\sqrt{1-x-x^2}}.$$
Finally we obtain
$$g(x)=\frac{4+6x^2+2x \sqrt{4+5x^2}}{4+5x^2+x\sqrt{4+5x^2}}.$$
Thus a horizontal antecedent Riordan array to the Fibonacci array $\left(\frac{1}{1-x-x^2},\frac{x}{1-x-x^2}\right)$ is given by the array
$$\left(\frac{4+6x^2+2x \sqrt{4+5x^2}}{4+5x^2+x\sqrt{4+5x^2}},\frac{1}{2}(x^2+x \sqrt{4+5x^2})\right).$$
This array begins
$$\left(
\begin{array}{ccccccccc}
 1 & 0 & 0 & 0 & 0 & 0 & 0 & 0 & 0 \\
 \frac{1}{2} & 1 & 0 & 0 & 0 & 0 & 0 & 0 & 0 \\
 0 & 1 & 1 & 0 & 0 & 0 & 0 & 0 & 0 \\
 -\frac{5}{16} & \frac{7}{8} & \frac{3}{2} & 1 & 0 & 0 & 0 & 0 & 0 \\
 0 & 0 & 2 & 2 & 1 & 0 & 0 & 0 & 0 \\
 \frac{75}{256} & -\frac{45}{128} & \frac{17}{16} & \frac{27}{8} & \frac{5}{2} & 1 & 0 & 0 & 0 \\
 0 & 0 & 0 & 3 & 5 & 3 & 1 & 0 & 0 \\
 -\frac{625}{2048} & \frac{275}{1024} & -\frac{95}{256} & \frac{203}{128} & \frac{95}{16} & \frac{55}{8}
   & \frac{7}{2} & 1 & 0 \\
 0 & 0 & 0 & 0 & 5 & 10 & 9 & 4 & 1 \\
\end{array}
\right).$$
The first column of this array, with generating function $g(x)$, begins
$$1,\frac{1}{2},0,-\frac{5}{16},0,\frac{75}{256},0,-\frac{625}{2048},0,\frac{21875}{65536},0,\ldots.$$
Its Hankel transform $h_n$ is given as follows.
\begin{proposition} The Hankel transform of the expansion of $g(x)$ is given by
$$h_n = (-1)^{\binom{n+1}{2}}\frac{ 5^{2 \lfloor \frac{n^2}{4} \rfloor}}{4^{n^2}}.$$
\end{proposition}
This follows from the following result.
\begin{proposition} The generating function $g(x)$ has the following Jacobi continued fraction expansion.
$$g(x)=
\cfrac{1}{1-\frac{x}{2}+
\cfrac{\frac{x^2}{4}}{1-\frac{3x}{4}+
\cfrac{\frac{25x^2}{16}}{1+x+
\cfrac{\frac{x^2}{16}}{1-x+
\cfrac{\frac{25x^2}{16}}{1+x+
\cfrac{\frac{x^2}{16}}{1-\cdots}}}}}}.$$
Here, the Jacobi parameters are
$$\mathcal{J}(1/2,3/4,-1,1,-1,1,-1,1,\ldots;-1/4,-25/16,-1/16,-25/16,-1/16,-25/16,\ldots).$$
\end{proposition}
Thus the elements of the expansion of $g(x)$ are the moments of a family of orthogonal polynomials.
\end{example}

\section{Acknowledgements}
The author completed this article while a guest at the Applied Algebra and Optimization Research Center (AORC) of Sungkyunkwan University, Suwon, South Korea. He gratefully acknowledges their hospitality.

\bigskip
\hrule

\noindent 2010 {\it Mathematics Subject Classification}: Primary
15B36; Secondary 11B83, 11C20.
\noindent \emph{Keywords:} Riordan group, Riordan array, central coefficients, Lagrange inversion.

\end{document}